\documentclass[12pt,reqno]{amsart}
\usepackage{amsmath}
\usepackage{amsfonts}
\usepackage{amssymb}
\usepackage{amsthm}
\usepackage{enumerate}

\newcommand{\z}{\mathbb{Z}}
\newcommand{\n}{\mathbb{N}}

\newtheorem{thm}{Theorem}
\newtheorem{lem}{Lemma}

\theoremstyle{remark}
\newtheorem*{rem}{Remark}

\theoremstyle{definition}

\author{Jingjing Huang}
\address{
Department of Mathematics, McAllister Building,
Pennsylvania State University, University Park, PA 16802-6401,
U.S.A.}
\email{huang@math.psu.edu}
\thanks{}

\begin{document}

\title
[$axy-x-y=n$]
{A Mean Value Theorem for the Diophantine Equation $axy-x-y=n$}

\begin{abstract}
In this paper, we prove an asymptotic formula for the average number of solutions to the Diophantine equation $axy-x-y=n$ in which $a$ is fixed and 
$n$ varies. 
\end{abstract}
\maketitle

\section{Introduction} 

\noindent People has been considering Diophantine equations involving products and sums of some variables for a long time.
The Diophantine equation
\begin{equation} \label{e10}
\prod_{i=1}^k x_i-\sum_{i=1}^k x_i=n
\end{equation}
was studied by various people during the past a few decades. It is easy to see that there always exists a few trivial solutions with most of $x_i$'s equal to 1. So people are asking about the number of solutions of this equation with all $x_i>1$. 
\par
The case when $n=0$ is very special, since it concerns the number of $k$-tuples with equal sum and product. In this case, it is conjectured by Misiurewicz \cite{mi} that $k=2, 3, 4, 6, 24, 114, 174$ and 444 are the only values of $k$ for which there are only trivial solutions. For general $n$, very little is known except that in 1970s Viola \cite{vi} proved that if $E_k(N)$ denotes the number of positive integers $n\le N$ for which \eqref{e10} is not soluble in integers $x_1, x_2, \ldots, x_k>1$ then $E_k(N)=N \exp(-c_k(\log N)^{1-1/(k+1)})$ for some positive constant $c_k$. It is believed that for large $n$ equation \eqref{e10} always has a nontrivial solution, which nevertheless is an open question in this area. 
\par
On the other hand, the case that $k=3$ has received extensive attention, and several variations of this problem were studied. Brian Conrey asked whether the number of solutions in positive integers to the equation
\begin{equation}\label{e11}
xyz+x+y=n
\end{equation}
can be bounded by $O_\varepsilon(n^\varepsilon)$ for any $\varepsilon>0$. Kevin Ford posed a generalisation of this problem, in which one would like to show that there are $O_\varepsilon(|AB|^\varepsilon)$ nontrivial positive integer solutions to the equation 
\begin{equation}\label{e12}
xyz=A(x+y)+B
\end{equation} 
for given nonzero $A, B\in\z$. 
\par
In this paper, we consider another variation of the case that $k=2$, namely the following equation
\begin{equation} \label{e1}
axy-x-y=n
\end{equation}
where $a$ is a positive integer and $n$ is any nonnegative integer. This can be viewed as equation \eqref{e12} in which $z$ is fixed and $A=1$. Hence if the number of solutions of equation \eqref{e1} is well understood, then one can probably understand the number of solutions of equation \eqref{e12} simply by averaging over $a$. 
\par
Let $$R_a(n)=\rm{Card}\left\{(x,y)\in\n^2:axy-x-y=n\right\}.$$ 
Here we are considering the number of positive integer solutions of equation \eqref{e1} when $a$ is fixed and $n$ varies. A sharp asymptotic formula is established in this paper on the average of $R_a(n)$ over $n$. Notice the case that $a=1$ is trivial, since then $R_1(n)=d(n+1)$ is just the divisor function of $n+1$, the average of which is relatively well understood. 

\begin{thm} \label{t1} For positive integers $a>1$ and $N\ge 1$, we have
$$\sum_{0\leq n\leq N}R_a(n)=\frac1a\Big(N\log N-C(a)N\Big)+\Delta_a(N)$$
where
\begin{equation} \label{e13}
C(a)=2\frac{\Gamma'{(\frac{a-1}a)}}{\Gamma(\frac{a-1}a)}+2\sum_{p|a} \frac{\log p}{p-1}+\log a+2\gamma+1
\end{equation} 
and
\begin{equation} \label{e14}
\Delta_a(N)\ll{\phi(a)}\sqrt{\frac{N}a}\Big(\log (aN)\Big)^2.
\end{equation}
Here $\Gamma(s)=\int_0^{+\infty}e^{-t}t^{s-1}d t$ is the standard $\Gamma$ function, and $\gamma$ is the Euler constant.
\end{thm}

In fact, since the error term above is roughly of size $\sqrt{aN}\Big(\log (aN)\Big)^2$, it is conceivable that the main term will be inferior to the error term when 
$a\gg N^{\frac{1}{3}}$. So in order for the above asymptotic formula to really make sense, one would impose a condition on $a$, such as 
$a\ll N^{\frac13}/\log N$.

Moreover, one can argue what is the right order of magnitude of the error $\Delta_a(N)$. In view of $R_1(n)=d(n+1)$, one can think $R_a(n)$ as a ``generalized" divisor function. Hence Theorem \ref{t1} just proves a mean value theorem for such a ``generalized" divisor function. Since for the classical divisor function, the error is believed to be $O(N^{1/4+\varepsilon})$. It is very natural to pose such a conjecture for our error $\Delta_a(N)$. The author suspects that following the van der Corput method on exponential sums as in the classical case, one can show $\Delta_a(N)=O_a(N^{1/3-\delta})$ for some $\delta>0$. 

\begin{rem}
It's not hard to adapt the method in this paper in order to deal with equations like
$$axy-bx-cy=n$$
and prove similar asymptotic formulas. 
\end{rem}

\section{Preliminary Lemmas}

We state several lemmas before embarking on the proof of Theorem 1.  The content of Lemma 2 
can be found, for example, in Corollary 1.17 and Theorem 6.7 of Montgomery \& Vaughan \cite{MV}, and Lemma 3 
can be deduced from Theorem 4.15 of Titchmarsh \cite{ti} with $x=y=(|t|/2\pi)^{1/2}$.
\newtheorem{lemma}[thm]{Lemma}
\begin{lem}
\label{l1}
When $\sigma\geq 1$ and $|t|\geq2$, we have $$\frac{1}{\log|t|}\ll \zeta(\sigma+it)\ll \log|t|.$$
\end{lem}

\begin{lem}
\label{l2}
When $0\leq\sigma\leq1$ and $|t|\geq2$, we have
$$\zeta(\sigma+it)\ll |t|^\frac{1-\sigma}{2}\log(|t|).$$
\end{lem}

\begin{lem}
\label{l3}
Let $\chi$ be a non-principle character modulo $a$ and $s=\sigma+it$ and assume that $t\in\mathbb R$.  Then
$$L(s,\chi)\ll \log(a(2+|t|)),\text{ when }\sigma\geq 1$$
and
$$L(s,\chi)\ll \left(a|t|\right)^{\frac{1-\sigma}{2}+\varepsilon},\text{ when }\frac{1}{2}\leq\sigma\leq1.$$
\end{lem}

\begin{proof}
The first part follows from Lemma 10.15 of MV \cite{MV}.  Now suppose that $\chi$ is primitive.  Then by Corollary 10.10 of MV \cite{MV}, 
$$L(s,\chi)\ll \left(a|t|\right)^{\frac12-\sigma}\log (a(2+|t|))$$
when $\sigma\le 0$.  Then by the convexity principle for Dirichlet series, for example as described in Titchmarsh \cite{ti} (cf. exercise 10.1.19 of MV \cite{MV}),
$$L(s,\chi)\ll \left(a|t|\right)^{\frac{1-\sigma}{2}+\varepsilon}$$
when $0\le \sigma\le 1$.  The proof is completed by observing that if $\frac12\le\sigma\le 1$ and $\chi$ modulo 
$a$ is induced by the primitive character $\chi^*$ with conductor $q$, then 
$$L(s,\chi) = 
L(s,\chi^*)\prod_{\substack{p|a\\ p\nmid q}}(1-\chi^*(p)p^{-s}) \ll |L(s,\chi^*)|2^{\omega(a)}.$$
\end{proof}

\begin{lem}
\label{l4}
Let $T\geq 2$, then we have
$$\sum_{\substack{\chi\\ \bmod{a}}}\int_{-T}^T|L(\frac{1}{2}+it,\chi)|^2d t\ll \frac{\phi^2(a)}{a} T \log T.$$ 
\end{lem}
A proof of this lemma can be found for example in Montgomery \cite{mo}.

\begin{lem}\label{l6} 
Let $a$ be a positive integer greater than 1 and $w>0$, we have
$$\sum_{\substack{n\leq w\\n\equiv-1\bmod{a}}}\frac1n=\frac1a\left(\log w-\frac{\Gamma'{(\frac{a-1}a)}}{\Gamma(\frac{a-1}a)}-\log a\right)+O(1/w).$$

\end{lem}
\begin{proof}
By Abel summation, the left hand side above is
\begin{eqnarray*}
\sum_{\substack{n\leq w\\n\equiv-1\bmod{a}}}\frac1n
&=&\left\lfloor\frac{w+1}a \right\rfloor\frac1w+\int_1^{w}\left\lfloor\frac{t+1}a \right\rfloor\frac1{t^2}d t\\
&=&\frac1a+\int_1^{w}\frac{t+1}{a t^2}d t-\int_1^{w}\left\{\frac{t+1}{a}\right\}\frac{d t}{t^2}+O(1/w)\\
&=&\frac1a\left(\log w+2-\int_1^\infty a\left\{\frac{t+1}{a}\right\}\frac{d t}{t^2}\right)+O(1/w).
\end{eqnarray*}
Recall that the digamma function $\psi(z)$ is defined as $\frac{\Gamma'}{\Gamma}(z)$, and $\psi'(z)$ has a series expansion $\sum_{k=0}^{\infty}\frac1{(z+k)^2}$. So 
\begin{eqnarray}
\int_1^\infty \left(a\left\{\frac{t+1}{a}\right\}-\{t\}-1\right)\frac{d t}{t^2}
&=&\sum_{h=0}^{\infty}\int_0^a\left(a\left\{\frac{r+1}{a}\right\}-\{r\}-1\right)\frac{d r}{(ah+r)^2} \nonumber \\
&=&\frac1{a^2}\int_0^a\left(a\left\{\frac{r+1}{a}\right\}-\{r\}-1\right)\psi'\left(\frac{r}{a}\right)d r \nonumber \\
&&\label{e15}
\end{eqnarray}
Notice that 
\[a\left\{\frac{r+1}{a}\right\}-\{r\}-1=\left\{
\begin{array}{ccc}
0, &\mbox{if}& 0\le r<1\\
1, &\mbox{if}& 1\le r<2\\
\vdots & &\vdots\\
a-2, &\mbox{if}&a-2\le r<a-1\\
-1, &\mbox{if}& a-1\le r<a
\end{array}
\right.
\]
Hence \eqref{e15} is equal to
\begin{eqnarray*}
&&\frac1{a^2}\left(\sum_{l=1}^{a-2}l\int_{0}^1\psi'\left(\frac{l+r}{a}\right)d r-\int_0^1\psi'\left(\frac{a-1+r}{a}\right)d r\right)\\
&=&\frac1a\left(\sum_{l=1}^{a-2}l\left(\psi\left(\frac{l+1}a\right)-\psi\left(\frac{l}a\right)\right)-\left(\psi(1)-\psi\left(\frac{a-1}{a}\right)\right)\right)\\
&=&\psi\left(\frac{a-1}a\right)-\frac1a\sum_{l=1}^{a}\psi\left(\frac{l}a\right)\\
&=&\psi\left(\frac{a-1}a\right)+\log a+\gamma
\end{eqnarray*}
The last equality follows from a well known property of the digamma function $\psi$. Now the lemma is established after the observation 
$\gamma=2-\int_1^{\infty}\frac{\{t\}+1}{t^2}d t$.
\end{proof}

\begin{lem}
\label{l5}
Let $a$ be a positive integer greater than 1, then we have
$$\frac1{\phi(a)}\sum_{\substack{\chi\not=\chi_0\\ \bmod a}} \Bar\chi(-1)L(1,\chi)=-\frac1{a}\left(\frac{\Gamma'{(\frac{a-1}a)}}{\Gamma(\frac{a-1}a)}+\sum_{p|a} \frac{\log p}{p-1}+\log a+\gamma\right) 
.$$
\end{lem}
\begin{proof}
Let $w$ be large compared to $a$ (eventually we will let $w$ goes to $\infty$).  Then for non-principal characters $\chi$ modulo $a$, by Abel summation 
$$L(1,\chi)=\sum_{n\le w}\frac{\chi(n)}n + O(a/w).$$ 
Hence
\begin{align*}
\frac1{\phi(a)}\sum_{\substack{\chi\not=\chi_0\\ \bmod a}}& \Bar\chi(-1)L(1,\chi)\\
&= \frac1{\phi(a)}\sum_{\substack{\chi\not=\chi_0\\ \bmod a}} \Bar\chi(-1)
\sum_{n\le w}\frac{\chi(n)}n
 +O(a/w).
\end{align*}
The main term on the right is
\begin{equation*}
\frac1{\phi(a)}\sum_{\substack{\chi\\ \bmod a}} \Bar\chi(-1)
\sum_{n\le w}\frac{\chi(n)}n
 - \frac1{\phi(a)}\sum_{\substack{n\le w\\(n,a)=1}}\frac1n.
\end{equation*}
We have
\begin{align*}
\sum_{\substack{n\le w\\(n,a)=1}}\frac1n&= \sum_{m|a}\frac{\mu(m)}m\sum_{n\le w/m}\frac1n\\
&=\sum_{m|a}\frac{\mu(m)}m\Big(\log(w/m)+\gamma+O(m/w)\Big)\\
&=\frac{\phi(a)}a\Bigg(\log w+\sum_{p|a}\frac{\log p}{p-1}+\gamma\Bigg) +O\Big(d(a)/w\Big).
\end{align*}
Here we are using the fact that $-\sum_{m|a}\frac{\mu(m)}m\log m=\frac{\phi(a)}a\sum_{p|a}\frac{\log p}{p-1}$, 
this is because
\begin{eqnarray*}
-\sum_{m|a}\frac{\mu(m)}m\log m&=&\sum_{p|a}\frac{\log p}{p}\sum_{\substack{k|a/p\\(p,k)=1}}\frac{\mu(k)}{k}\\
&=&\sum_{p|a}\frac{\log p}{p}\prod_{\substack{p'|a\\p'\not=p}}\left(1-\frac{1}{p'}\right)\\
&=&\sum_{p|a}\frac{\log p}{p}\left(\frac1{1-\frac{1}{p}}\right)\prod_{p'|a}\left(1-\frac{1}{p'}\right)\\
&=&\frac{\phi(a)}a\sum_{p|a}\frac{\log p}{p-1}.
\end{eqnarray*}
On the other hand, we have
$$\frac1{\phi(a)}\sum_{\substack{\chi\\ \bmod a}} \Bar\chi(-1)
\sum_{n\le w}\frac{\chi(n)}n=\sum_{\substack{n\leq w\\n\equiv-1\bmod{a}}}\frac1n$$
And by lemma \ref{l6}, this is
$$\frac1a\left(\log w-\frac{\Gamma'{(\frac{a-1}a)}}{\Gamma(\frac{a-1}a)}-\log a\right)+O(1/w).$$
Thus we have shown that
\begin{eqnarray*}
& &\frac1{\phi(a)}\sum_{\substack{\chi\not=\chi_0\\ \bmod a}} \Bar\chi(-1)L(1,\chi)\\
&=& \frac{1}{a}\Big(\log w-\frac{\Gamma'{(\frac{a-1}a)}}{\Gamma(\frac{a-1}a)}-\log a\Big)-\frac1a\Bigg(\log w+\sum_{p|a}\frac{\log p}{p-1}+\gamma\Bigg) +O(a/w)\\
&=&-\frac1{a}\left(\frac{\Gamma'{(\frac{a-1}a)}}{\Gamma(\frac{a-1}a)}+\sum_{p|a} \frac{\log p}{p-1}+\log a+\gamma\right)+O(a/w)
\end{eqnarray*}
Now the lemma is established when we let $w\rightarrow\infty$ in the above.
\end{proof}

\section{Proof of Theorem \ref{t1}}
\noindent 
The starting point of the proof is the following observation. One can rewrite equation (\ref{e1}) in the following form
\begin{equation}
\label{e2}
(ax-1)(ay-1)=an+1.
\end{equation}
Namely we are going to count the following quantities, 
$$R_a(n)=\mathrm{Card}\left\{(x,y)\in\n^2:(ax-1)(ay-1)=an+1\right\}$$  
and
$$S_a(N)=\sum_{0\leq n\leq N}R_a(n).$$

After the change of variables $u=ax-1$ and $v=ay-1$, it follows that $R_a(n)$ is the number of ordered pairs of natural numbers $u$, $v$ such that $uv=an+1$ and $u\equiv v\equiv -1\pmod{a}$.\par
Now the residue class $u\equiv -1 \pmod{a}$ and $v\equiv -1 \pmod{a}$ are readily isolated {\it via} the orthogonality of the Dirichlet characters $\chi$ modulo $a$.  Thus we have
\begin{eqnarray*}
&&S_a(N)\\
&=&\sum_{0\leq n\leq N}\sum_{\substack{ uv=an+1\\u\equiv -1\bmod{a}\\v\equiv -1 \bmod{a}}}1\\
&=&\sum_{m\leq M}\sum_{\substack{uv=m\\u\equiv -1\bmod{a}\\v\equiv -1 \bmod{a}}}1\\
&=&\frac{1}{\phi^2(a)}\sum_{\substack{\chi_1\\\bmod{a}}}\sum_{\substack{\chi_2\\\bmod{a}}} \bar\chi_1(-1)\bar{\chi}_2(-1)\sum_{m\leq M}\sum_{uv=m}{\chi}_1(u)\chi_2(v),
\end{eqnarray*}
where $M=aN+1$.

Let 
\begin{equation*}
a_m(\chi_1,\chi_2)=\displaystyle\sum_{uv=m}{\chi}_1(u)\chi_2(v).
\end{equation*} 
Then we have
\begin{equation*} 
S_a(N)=\frac{1}{\phi^2(a)}\sum_{\substack{\chi_1\\\bmod{a}}}\sum_{\substack{\chi_2\\\bmod{a}}} \bar\chi_1(-1)\bar{\chi}_2(-1)\sum_{m\leq M}a_m(\chi_1,\chi_2).
\end{equation*}
We analyze this expression through the properties of the Dirichlet series 
\begin{equation} \label{e5}
f_{\chi_1,\chi_2}(s)=\displaystyle\sum_{m=1}^{\infty}\frac{a_m(\chi_1,\chi_2)}{n^s}=L(s,\chi_1)L(s,\chi_2).
\end{equation}
This affords an analytic continuation of $f_{\chi_1,\chi_2}$ to the whole complex plane.

By a quantitative version of Perron's formula, as in Theorem 5.2 of MV \cite{MV} for example,  we obtain

$$\sideset{}{'}\sum_{m\leq M}a_m(\chi_1,\chi_2)=\frac{1}{2\pi i}\int_{\sigma_0-iT}^{\sigma_0+iT}f_{\chi_1,\chi_2}(s)\frac{M^s}{s}ds+R(\chi_1,\chi_2),$$
where $\sigma_0>1$ and 
\begin{eqnarray*}
R(\chi_1,\chi_2)\ll&\sum_{\substack{\frac{M}{2}<m<2M\\m\neq M}}|a_m(\chi_1,\chi_2)|\min\left(1,\frac{M}{T|m-M|}\right)\\
&+\frac{4^{\sigma_0}+M^{\sigma_0}}{T}\sum_{m=1}^{\infty}\frac{|a_m(\chi_1,\chi_2)|}{m^{\sigma_0}}.
\end{eqnarray*}
Here $\sideset{}{'}\sum$ means that when $M$ is an integer, the term $a_M(\chi_1,\chi_2)$
is counted with weight $\frac{1}{2}$.

Let $\sigma_0=1+\frac{1}{\log M}$.  By (\ref{e5}) we have $|a_m(\chi_1,\chi_2)|\le d(n)$.  Thus
$$\sum_{m=1}^{\infty}\frac{|a_m(\chi_1,\chi_2)|}{n^{\sigma_0}}\ll \zeta(\sigma_0)^2\ll (\log B)^2$$
and so $R(\chi_1,\chi_2)\ll_\varepsilon M^{1+\varepsilon}T^{-1}$, for any
$\varepsilon>0$.  Hence
$$\sum_{m\leq M}a_m(\chi_1,\chi_2)=\frac{1}{2\pi i}\int_{\sigma_0-iT}^{\sigma_0+iT}f_{\chi_1,\chi_2}(s)\frac{M^s}{s}ds+O\left(\left(\frac{M}{T}+1\right)M^\varepsilon\right).$$
The error term here is
$$\ll M^{\varepsilon}$$
provided that
$$T\ge M.$$
The integrand is a meromorphic function in the complex plane and is analytic for all $s$ with $\Re s\ge\frac12$ except for a possible pole of finite order at $s=1$.  Suppose that $T\ge4$.  By the residue theorem we have
\begin{align*}
&\frac{1}{2\pi i}\int_{\sigma_0-iT}^{\sigma_0+iT}f_{\chi_1,\chi_2}(s)\frac{M^s}{s}ds\\
=&\frac{1}{2\pi i}\left(\int_{\sigma_0-iT}^{\frac{1}{2}-iT}+\int_{\frac{1}{2}-iT}^{\frac{1}{2}+iT}+\int_{\frac{1}{2}+iT}^{\sigma_0+iT}\right)
\frac{L(s,\chi_1)L(s,\chi_2) M^s}{s}ds\\
&+\text{Res}_{s=1}\left(L(s,\chi_1)L(s,\chi_2)\frac{M^s}{s}\right).
\end{align*}

Hence, by Lemmas 1, 2 and 3, the contribution from the horizontal paths is 
\begin{align*}
&\ll (\log aT)^2\frac{M}{T\log M} +\frac{(aT)^{\varepsilon}}{T}\int_{1/2}^1 (aT)^{1-\sigma}M^{\sigma}d\sigma\\
&\ll T^{-1}(aT)^{\varepsilon}M+T^{-1}(aT)^{1/2+\varepsilon}M^{1/2}
\end{align*}
and provided that $T\ge M^{5}$ this is 
$$\ll M^{-1}.$$
\par
On the other hand, the contribution from the vertical path on the right is bounded by
$$M^{\frac12}\sum_{2^k\le T}2^{-k}\int_{2^{k}}^{2^{k+1}}\textstyle|L(\frac12+it,\chi_1)L(\frac12+it,\chi_2)|dt.$$
And by Lemma 4 
\begin{align*}\sum_{\substack{\chi_1,\chi_2\\ \bmod a}}\bar\chi_1(-1)\bar\chi_2(-1) \frac{1}{2\pi i} & \int_{\frac{1}{2}-iT}^{\frac{1}{2}+iT} 
\frac{L(s,\chi_1)L(s,\chi_2) M^s}{s}ds\\
& \ll M^{\frac12}\sum_{2^k\le T} 2^{-k}\int_{2^{k}}^{2^{k+1}}\bigg(\sum_{\substack{\chi\\\bmod a}}\textstyle|L(\frac12+it,\chi)|\bigg)^2dt\\
& \ll M^{\frac12}\sum_{2^k\le T} 2^{-k}\phi(a)\sum_{\substack{\chi\\ \bmod a}}\int_{-2^{k+1}}^{2^{k+1}}\textstyle|L(\frac12+it,\chi)|^2dt\\
& \ll M^{\frac12}\sum_{2^k\le T} \frac{\phi^3(a)}{a} k\\
& \ll \frac{\phi^3(a)}{a} M^{\frac12}(\log M)^2
\end{align*}
on taking 
$$T=M^{5}.$$
Hence we obtain
\begin{align*}
S_a(N)=\frac{1}{\phi^2(a)}\sum_{\substack{\chi_1\\ \mod a}}\sum_{\substack{\chi_2\\ \mod a}}\bar\chi_1(-1)\bar\chi_2(-1)\text{Res}_{s=1}\left(f_{\chi_1,\chi_2}(s)\frac{M^s}{s}\right)
 + \Delta_a(N)
\end{align*}
where
\begin{equation}
\Delta_a(N)\ll \frac{\phi(a)}a\sqrt{M}(\log M)^2\ll{\phi(a)}\sqrt{\frac{N}a}\Big(\log (aN)\Big)^2.
\end{equation}
It remains to compute the residue at $s=1$.\par

By (\ref{e5}) there are naturally two cases, namely
\begin{enumerate}[(i)] 
\item $\chi_1=\chi_2=\chi_0$;
\item only one of $\chi_1$ and $\chi_2$ is equal to $\chi_0$ while the other one is equal to $\chi\neq\chi_0$. 
\end{enumerate}
In the latter case the integrand has a simple pole at $s=1$ and the residue is
$$\prod_{p|a}\left(1-\frac1p\right)L(1,\chi)(a N+1)=\phi(a)L(1,\chi) N+\frac{\phi(a)}{a}L(1,\chi).$$
By lemma \ref{l5}, the sum over $\chi$ for the second term above is small, hence can be absorbed in $\Delta_a(N)$.
While in the former case, the integrand has a double pole at $s=1$ and the residue is 
$$\prod_{p|a}\left(1-\frac1p\right)^2\Big(M\log M-M\Big).$$
Hence we have shown that
\begin{eqnarray*} 
S_a(N)=&\frac{1}{a^2}\Big((aN+1)\log (aN+1)-aN-1\Big)\\
+&\displaystyle\Big(\frac2{\phi(a)}\sum_{\substack{\chi\neq\chi_0\\\mod a}}\bar\chi(-1)L(1,\chi)\Big) N+\Delta_a(N).
\end{eqnarray*}
Now by lemma \ref{l5}, this is
$$\frac1a\Big(N\log N-C(a)N\Big)+\Delta_a(N)$$
where $C(a)$ and $\Delta_a(N)$ are given by \eqref{e13} and \eqref{e14} respectively.

This completes the proof of Theorem \ref{t1}.

\hfill $\square$

\end{document}